\documentclass[12pt,a4paper]{amsart}
\usepackage{amssymb,stmaryrd}
\usepackage[a4paper,left=2.25cm,right=2.25cm,top=3cm,bottom=3cm,headsep=1cm]{geometry}
\usepackage[mathbf]{euler} 
\usepackage[applemac]{}
\usepackage{bm}
\usepackage{blindtext}
\usepackage{amsmath}
\usepackage[toc,page]{appendix} 
\usepackage{amsfonts}
\usepackage{amsthm}
\usepackage{caption}
\usepackage{color}
\usepackage{shadethm}
\usepackage{subfigure}
\usepackage{amsmath,amscd,tikz-cd, mathrsfs, amssymb, tikz, listings}
\usepackage{amssymb,amsthm, xcolor}
\usepackage[
backend=biber,
style=numeric,
sorting=nyt
]{biblatex}
\usepackage[utf8]{inputenc}
\usepackage{csquotes}
\usepackage[english]{babel}
\usepackage[colorlinks = true,
            linkcolor = blue,
            urlcolor  = blue,
            citecolor = blue,
            anchorcolor = blue]{hyperref}

\usepackage{epigraph}
\usepackage{listings}
\usepackage{color}

\definecolor{dkgreen}{rgb}{0,0.6,0}
\definecolor{gray}{rgb}{0.5,0.5,0.5}
\definecolor{mauve}{rgb}{0.58,0,0.82}

\usepackage{enumerate}
\usepackage{verbatim}
\usetikzlibrary{trees}
\usetikzlibrary{decorations.markings}
\usetikzlibrary{arrows}
\usepackage{mathrsfs}
\usepackage{url}

\makeatletter
\def\@tocline#1#2#3#4#5#6#7{\relax
  \ifnum #1>\c@tocdepth 
  \else
    \par \addpenalty\@secpenalty\addvspace{#2}%
    \begingroup \hyphenpenalty\@M
    \@ifempty{#4}{%
      \@tempdima\csname r@tocindent\number#1\endcsname\relax
    }{%
      \@tempdima#4\relax
    }%
    \parindent\z@ \leftskip#3\relax \advance\leftskip\@tempdima\relax
    \rightskip\@pnumwidth plus4em \parfillskip-\@pnumwidth
    #5\leavevmode\hskip-\@tempdima
      \ifcase #1
       \or\or \hskip 1em \or \hskip 2em \else \hskip 3em \fi%
      #6\nobreak\relax
    \hfill\hbox to\@pnumwidth{\@tocpagenum{#7}}\par
    \nobreak
    \endgroup
  \fi}
\makeatother

\theoremstyle{plain}
   \newtheorem{theorem}{Theorem}[section]

   \newtheorem{corollary}[theorem]{Corollary}
   \newtheorem{porism}[theorem]{Porism}
   
\theoremstyle{definition}
    \newtheorem{definition}[theorem]{Definition}
    \newtheorem{conjecture}{Conjecture}
    \newtheorem{question}{Question}
    
    \newtheorem{lemma}[theorem]{Lemma}
\newtheorem{example}[theorem]{Example}
\definecolor{green}{RGB}{34, 139, 34}

\theoremstyle{remark}
\newtheorem{remark}{Remark}

\newcommand{\MM}{\mathcal{M}}

\newcommand{\NN}{\mathbb{N}}

\newcommand{\ZZ}{\mathbb{Z}}

\newcommand{\RR}{\mathbb{R}}

\renewcommand{\bar}{\overline}

\newcommand{\im}{\text{im }}

\renewcommand{\phi}{\varphi}
\newcommand{\supp}{\text{supp}}

\renewcommand{\emph}{\textit}
\newcommand{\at}{\color{red}}
\newcommand{\rmat}{\color{black}}

\makeatletter

\pgfdeclareshape{genus}{
 \anchor{center}{\pgfpointorigin}
\backgroundpath{
    \begingroup
    \tikz@mode
    \iftikz@mode@fill

         \pgfpathmoveto{\pgfqpoint{-0.78cm}{-.17cm}}
     \pgfpathcurveto %
            {\pgfpoint{-0.35cm}{-.44cm}}
        {\pgfpoint{0.35cm}{-.44cm}}
        {\pgfpoint{.78cm}{-0.17cm}} 
     \pgfpathmoveto{\pgfqpoint{-0.78cm}{-0.17cm}}
     \pgfpathcurveto %
            {\pgfpoint{-0.25cm}{.25cm}}
        {\pgfpoint{.25cm}{.25cm}}
        {\pgfpoint{0.78cm}{-0.17cm}}
        \pgfusepath{fill}
        \fi

        \iftikz@mode@draw
     \pgfpathmoveto{\pgfqpoint{-1cm}{0cm}}
     \pgfpathcurveto %
            {\pgfpoint{-0.5cm}{-.5cm}}
        {\pgfpoint{0.5cm}{-.5cm}}
        {\pgfpoint{1cm}{0cm}}

     \pgfpathmoveto{\pgfqpoint{-0.75cm}{-0.15cm}}
     \pgfpathcurveto %
            {\pgfpoint{-0.25cm}{.25cm}}
        {\pgfpoint{.25cm}{.25cm}}
        {\pgfpoint{0.75cm}{-0.15cm}}
              \pgfusepath{stroke}
        \fi
        \endgroup
    }
    }

    \makeatother

\graphicspath{{figures/}}
\PassOptionsToPackage{linktocpage}{hyperref} 

\begin{document}

\title{On discrete gradient vector fields and Laplacians \linebreak of simplicial complexes}
\author[I. Contreras]{Ivan Contreras}
\author[A. R. Tawfeek]{Andrew R. Tawfeek}
\email[I.~Contreras]{icontreraspalacios@amherst.edu}
\email[A. R. ~Tawfeek]{atawfeek@uw.edu}
\maketitle

\begin{abstract}
Discrete Morse theory, a cell complex-analog to smooth Morse theory allowing homotopic tools in the discrete realm, has been developed over the past few decades since its original formulation by Robin Forman in 1998. In particular, discrete gradient vector fields on simplicial complexes capture important topological features of the structure. We prove that the characteristic polynomials of the Laplacian matrices of a simplicial complex are generating functions for discrete gradient vector fields if the complex is a triangulation of an orientable manifold. Furthermore, we provide a full characterization of the correspondence between rooted forests in higher dimensions and discrete gradient vector fields.
\end{abstract}


\section{Introduction}
Since the appearance of discrete Morse theory due to Forman in \cite{forman}, it has brought a wide range of applications of this combinatorial model of Morse theory to mathematics and the sciences alike. Some of the recent applications include cellular sheaf cohomology \cite{sheaf}, topology of random points \cite{randomcomplexes}, and studying human speech from statistical data \cite{humanspeech}.

Here, we further demonstrate the importance of discrete Morse theory by showing that discrete gradient fields have a novel intimate connection with the topological structure of a simplicial complex as characterized by its Laplacian $\Delta_d :=\partial_d\partial_d^T$. The discrete Laplacian plays an important role in various subjects such as chip-firing on graphs \cite{chip}, critical groups in higher dimensions \cite{critgroups}, and has even been generalized to matroids \cite{matroidlap}.

In particular, we prove (Theorem \ref{main}) that when $K$ is a triangulation of an orientable $d$-manifold, we have that
\begin{equation*}
    det( \Delta_d + \lambda I) = \sum_{i=0}^{n} |\mathcal{M}_{m-i} (K)| \lambda^{n-i},
\end{equation*}
where $m = |K_d|$ and $n = |K_{d-1}|$. Here, $\MM_\ell(K)$ denote the set of gradients on $K$ having $\ell$-many critical cells in dimension $d$ and all cells critical in dimension $n$ for $n < d-1$. In the graph-theoretic case the above statement holds for simple graphs (Theorem \ref{charpoly}).

\

This paper is organized as follows. In Section 2 we review the background and notation concerning simplicial complexes, higher dimensional rooted forests, and discrete Morse theory. In Section 3 we prove a special case of our main statement exclusively in the $1$-dimensional case, providing a graph theoretic interpretation of the main result, and motivating our methods in the general case. In Section 4, we conclude by proving the general statement, and then we present some future directions in Section 5.

\vspace{.5in}

\subsection*{Keywords} discrete Laplacian, simplicial complexes, discrete Morse theory, discrete gradient vector fields, matchings, rooted forests, spectral graph theory.
\section*{Acknowledgements}
We particularly thank Alejandro Morales for his valuable suggestions, comments and guidance during the project. We would also extend our thanks to David A. Cox, Michael Ching, Nathan Pflueger and Nicholas A. Scoville, for very helpful discussions and suggestions. The second author was partially supported by the Gregory S. Call Research Fellowship.

\section{Background and Notation}

\subsection{Simplicial complexes} A \textit{simplicial complex}\footnote{All simplicial complexes considered are finite-dimensional and connected, unless stated otherwise.} on $n$ vertices is a collection $K$ of subsets of $[n] := \{1,2, \dots, n\}$ which is closed under taking subsets, i.e.\ if $f \in K$ and $g \subseteq f$, then $g \in K$. An element $f \in K$ is called a $d$-dimensional \textit{face} or \textit{cell} (or $d$-face and $d$-cell) if $|f| = d + 1$. We use $K_d$ to denote the subset of $K$ consisting of only $d$-cells and $K_{\leq d}$ to denote the collection of all cells of dimension $\leq d$, called the $d$-\textit{skeleton}. When $K$ consists of all subsets of $[n]$, we call it a $n$-simplex, and denote it by $\Delta^{n}$. If $g \subsetneq f$ in $K$, we say $g$ is a \textit{face} of $f$ and write $g<f$ (with the notation chosen due to the later discussed poset structure on cells).  Cells $f \in K$ that are contained in no other cell of $K$ are said to be \textit{maximal}, and we say $K$ is an $n$-dimensional simplicial complex if its maximal cells are at most $n$-dimensional. For instance, a \textit{graph} is a $1$-dimensional simplicial complex.

Each $n$-dimensional simplicial complex has a \textit{geometric realization} as a subspace of $\mathbb{R}^{n+1}$, in which each $d$-face is represented by the convex hull of $d+1$ affinely independent points.

\begin{center}
    \captionsetup{type=figure}
    \includegraphics[width=.75\linewidth]{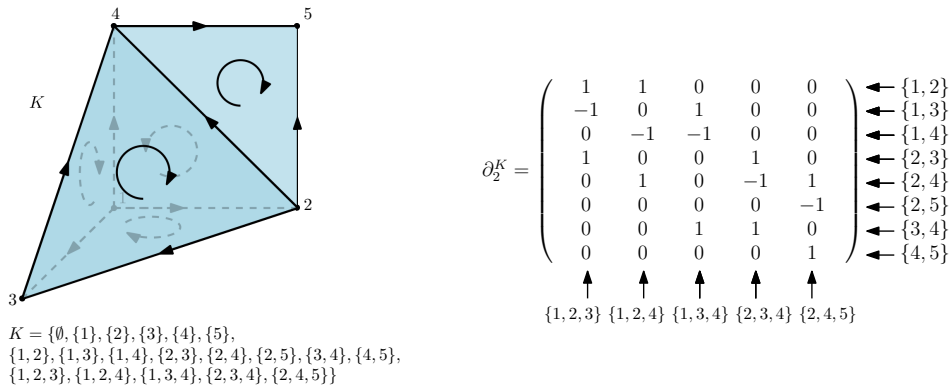} \label{mobfig}
    \captionof{figure}{An example of a $2$-dimensional simplicial complex $K$, with the standard lexicographic orientations indicated, and the the corresponding incidence matrix $\partial^K_2$.} 
\end{center}

To each simplicial complex $K$ we associate a chain complex $C_\bullet(K)$ of abelian groups\footnote{We will always consider $\ZZ$-modules.}
\[\begin{tikzcd}
	\cdots & {C_{d+1}(K)} & {C_d(K)} & {C_{d-1}(K)} & \cdots
	\arrow["{\partial^K_d}", from=1-3, to=1-4]
	\arrow[from=1-4, to=1-5]
	\arrow[from=1-1, to=1-2]
	\arrow["{\partial^K_{d+1}}", from=1-2, to=1-3]
\end{tikzcd}\]

\noindent where we let $C_d(K)$, for $d \geq 0$, denote the free abelian group generated by all orientations of all simplices of $K$ of dimension $d$, under the quotient by the relation that the sum of two opposite orientations is zero. The map $\partial^K_d$ (or just $\partial_d$ if no ambiguity will arise) is called the \textit{$d$th}\textit{-boundary map}, which we routinely identify with its corresponding matrix, called the \textit{$d$th}\textit{-incidence matrix} of the simplicial complex $K$, defined as follows:
\begin{itemize}
    \item the rows of $\partial_d$ are indexed by the $(d-1)$-faces and the columns are indexed by the $d$-faces,
    \item the entry $\partial_d[r,f]$ corresponding to a $(d-1)$-face $r$ and a $d$-face $f$ is $(-1)^j$ if $f = \{v_0,v_1,\dots,v_d\}$ with $v_0 < v_1 <\cdots <v_d$ and $r = f \setminus \{v_j\}$ for some $j$, and $0$ otherwise.
\end{itemize}

It is a standard exercise to show that $\partial_{d-1}\circ\partial_d = 0$. Elements of $C_n(K)$ are said to be \textit{$n$-chains}. A $n$-cycle of $K$ is a non-zero element of $C_n(K)$ that is in $\ker \partial_n$. Since for any $n \geq 1$ the image of an element $\tau \in C_n(K)$ is necessarily sent to the kernel of $\partial_{n-1}$, we say a $(n-1)$-cycle is a \textit{$K$-boundary} if it is inside $\im \partial_n$. 

We define $H_d(K)$ to be the $d$th integral homology group of the simplicial complex $K$ (or equivalently, of the chain complex $C_\bullet(K)$), given by
$$H_{d}(K) = \frac{\ker \partial_d}{\im \partial_{d+1}}.$$
This group is in-fact finitely generated, so we have by the fundamental theorem of finitely generated abelian groups that there exists primes $p_1,\dots, p_s$ and positive integers $n_1, \dots, n_s$ such that
$$H_d(K) \cong \mathbb{Z}^r \times \mathbb{Z}/{p_1^{n_1}}\ZZ \times \cdots \times \mathbb{Z}/{p_s^{n_s}}\ZZ,$$
and that this decomposition is uniquely determined by the group $H_d(K)$ (up to permutation of the factors). The value $r$ is called the \textit{rank} or \textit{Betti number}, denoted by $\beta_d(K)$ (or just $\beta_d$, if no ambiguity would arise). The product of the corresponding finite groups is called the \textit{torsion} of $H_d(K)$. If $H_d(K)$ has no finite subgroups, we say $K$ is \textit{torsion-free} in dimension $d$, or that $H_d(K)$ is \textit{torsion-free}.


Similarly, we have that the $d$th homology group of $K$ \textit{relative to a subcomplex} $A$ is as follows: given a subcomplex $A \subseteq K$, as the boundary maps $\partial^K_n: C_n(K) \to C_{n-1}(K)$ of $C_\bullet(K)$ leave $C_\bullet(A)$ invariant\footnote{That is to say, $\partial_n^K|_{C_{n}(A)}$ maps $C_{n}(A)$ to $C_{n-1}(A)$, and hence is precisely (after a restriction of the codomain of $\partial_n^K|_{C_{n}(A)}$ to $C_{n-1}(A) \subseteq C_{n-1}(K)$) $\partial_n^A$ of $C_\bullet(A)$.}, this induces the a short exact sequence 
$$0 \rightarrow C_\bullet(A) \rightarrow C_\bullet(K) \rightarrow C_\bullet(K)/C_\bullet(A) \rightarrow 0$$
between the chain complexes, where the boundary maps of $C_\bullet(K)/C_\bullet(A)$, 
$$\partial^*_n: C_{n}(K)/C_{n}(A) \to C_{n-1}(K)/C_{n-1}(A),$$
are given by $\partial^*_n(\sigma + C_{n}(A)) = \partial_n(\sigma) + C_{n-1}(A)$. We classically denote $C_\bullet(K)/C_\bullet(A)$ by $C_n(K,A)$, and write $H_d(K,A)$ to denote the $d$th homology group of $K$ relative to $A$, defined\footnote{For the readers familiar with CW complexes, a perhaps more geometrically intuitive way to visualize this group is to consider $K$ to be a CW complex, in which case due to the equivalence of cellular and simplicial homology (see \cite{hatcher}), we have that $H_d(K,A) \cong H_d(K/A)$ for $d \geq 1$ and $H_0(K,A) \cong H_0(K/A) \oplus \ZZ$.} to be the $d$th homology group of the chain complex $C_\bullet(K,A)$.

Let us now introduce the object of our main interest: the \textit{$d$th Laplacian} of a simplicial complex $K$ of dimension $d$, defined to be the $|K_{d-1}| \times |K_{d-1}|$ -matrix
$$\Delta^K_d = \partial_d \cdot \partial_d^T.$$
The rows and columns of $\Delta^K_d$ (or just $\Delta_d$, if no ambiguity would arise) are indexed by $(d-1)$-faces. The $(r,s)$th entry is given as follows.

$$\Delta_d^K[r,s] = 
\begin{cases}
\left| \coprod_{\sigma \in K_d} \sigma \cap \{r\} \right| &\text{if $r=s$}\\
(-1)^{i+j} &\text{if $r \neq s$ and $\exists f \in K_d$ such that $r = f \setminus \{v_i\}$ and $s = f \setminus \{v_j\}$}\\
0 &\text{otherwise ($r,s$ not incident to a common $d$-face).}
\end{cases}$$
In the case that $K$ is a $1$-dimensional simplicial complex, i.e.\ a graph, one often writes $\Delta_1$ as just $\Delta$ and explicitly calls it the \textit{graph Laplacian}.

\subsection{Rooted forests}

We now overview some of the basic notions on rooted forests, as introduced in \cite{higherdimforests}, which will be important for obtaining our higher dimensional results.

\begin{definition}
Let $K$ be a $d$-dimensional complex, and let $F$ be a subset of the $d$-faces. We call $F$ a \textit{forest} if the corresponding columns of $\partial_d^K$ are independent.
\end{definition}

Equivalently, note that a collection of $d$-face $F$ of a $d$-dimensional simplicial complex $K$ is a forest if and only if $\partial^F_d$ is injective.

\begin{remark} \label{matroids}
As pointed out in \cite{higherdimforests}, there is a connection between forests and matroids. For a $d$-dimensional simplicial complex $K$, the linear matroid of $\partial^K_d$, called the simplicial matroid $M_K$ of $K$, is the matroid having the ground set of $d$-faces and independent sets corresponding to forests. The bases of $M_K$ are the spanning forests of $K$: the $d$-faces corresponding to columns of $\partial^K_d$ that form a basis of the columns.
\end{remark}

\begin{definition}
Let $K$ be a $d$-dimensional complex. Let $R$ be a subset of the $(d-1)$-faces, and let $\bar{R}=K_{d-1}\setminus R$. We say that $R$ is \textit{relatively-free} if the rows of $\partial_d^K$ corresponding to the faces in $\bar{R}$ are of maximal rank. We say that $R$ is \textit{relatively-generating} if the rows of $\partial^K_d$ corresponding to the faces of $\bar{R}$ are independent. We say $R$ is a \textit{root} of $K$ if it is both relatively-free and relatively-generating, that is, the rows of $\partial_d^K$ corresponding to the faces in $\bar{R}$ form a basis of the rows.
\end{definition}

A more geometric view is as follows: given a subset of the $(d-1)$-faces of a $d$-dimensional complex $K$, we have that $R$ is \textit{relatively-free} if and only if $R$ contains no $K$-boundary (i.e. $C_d(R)$ contains no $K$-boundary). In a similar vein, $R$ is \textit{relatively-generating} if and only if for all cells $\sigma \in \bar{R}$, we have that $R \cup \{\sigma\}$ contains a $K$-boundary.

The conditions of being relatively-free and relatively-generating are natural generalizations of the properties of roots in the graph-theoretical case. A root $R$ being relatively-free generalizes the condition of containing \textit{at most} one vertex per connected component of a forest, and $R$ being relatively-generating generalizes the condition of having \textit{at least} one vertex per connected component in a forest.

\begin{definition}
Let $K$ be a $d$-dimensional simplicial complex. A \textit{rooted forest} of $K$ is a pair $(F,R)$ where $F$ is a $d$-dimensional forest of $K$, and $R$ is a root of $F$. 
\end{definition}

Lastly, given a rooted forest $(F,R)$, a \textit{fitting orientation} of $(F,R)$ is a bijection $\psi$ between $\bar{R} := K_{d-1}\setminus R$ and $F$, such that each face $r \in \bar{R}$ is mapped to a face $f \in F$ containing $r$. This generalizes the (unique) fitting orientation on a rooted forest of a graph, where each edge is oriented towards the root of its connected component.

\

\begin{center}
    \captionsetup{type=figure}
    \includegraphics[width=.85\linewidth]{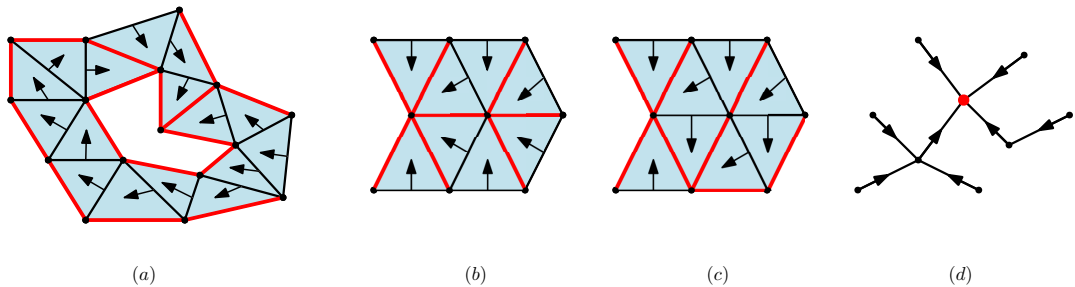} \label{rforest_ex}
    \captionof{figure}{Examples of rooted forests. The red $1$-cells in figures (a) -- (c) and red vertex in (d) denote the roots of the corresponding forest. The arrows represent fitting orientations.} 
\end{center}

\

A classic result of algebraic graph theory is that the Laplacian and rooted forests are intimately connected. Let $\Gamma = (V_\Gamma, E_\Gamma)$ be a graph on $n$ vertices having $m$ edges. Recall that an \textit{edge-subgraph} of $\Gamma$ is constructed by taking a subset $S \subseteq E_\Gamma$ together with all the vertices of $\Gamma$ incident with some edge belonging to $S$.

\begin{theorem}[Biggs 1974]\label{biggsthm}
The coefficients $q_i$ of the characteristic polynomial $\det(\Delta- \lambda I) = \sum_{i=0}^{n} (-1)^i q_{i}\lambda^{n-i}$, where $\Delta$ is the graph Laplacian, are given by the formula
$$q_i = \sum_{F \in \Omega_i} |V_F| \qquad (1 \leq i \leq n),$$
where $\Omega_i$ consists of all the edge-subgraphs of $\Gamma$ which have $i$ edges and are forests.
\end{theorem}

The above theorem also implies \textit{Kirchoff's Theorem}, i.e.\ the coefficient of $\lambda^1$ in $\det(\Delta- \lambda I)$ is subsequently
$$\lambda_1 \cdots \lambda_{n} = |V_\Gamma| \ \tau(\Gamma),$$
where the $\lambda_i$ above are the non-zero eigenvalues of the graph Laplacian and $\tau(\Gamma)$ is the number of spanning trees of $\Gamma$. The following theorem from \cite{higherdimforests} generalizes Theorem \ref{biggsthm} to higher dimensions.

\begin{theorem}[Bernardi-Klivans 2016] \label{higherlap}
For a $d$-dimensional simplicial complex $K$, the characteristic polynomial of the Laplacian matrix $\Delta_d$ gives a generating function for the rooted forests of $K$. More precisely,
$$\sum_{\text{$(F,R)$ rooted forest of $K$}} |H_{d-1}(F,R)|^2 \lambda^{|R|} = \det(\Delta_d + \lambda I),$$
where $I$ is the identity matrix of dimension $|K_{d-1}|$.
\end{theorem}

\subsection{Discrete Morse theory} 
Lastly, we overview the foundations of discrete Morse theory -- with a particular emphasis placed on concepts relevant to discrete gradient vector fields.

\begin{definition} \label{dmtfn}
Let $K$ be a simplicial complex. A \textit{discrete Morse function} on $K$ is a function $f:K \to \RR$ satisfying the following two inequalities for all $\sigma \in K_p$:
$$|\{ \tau \in K_{p-1} \ | \ \tau<\sigma \text{ and } f(\sigma)\leq f(\tau) \}| \leq 1;$$
$$|\{ \tau \in K_{p+1} \ | \ \sigma < \tau \text{ and } f(\tau)\leq f(\sigma) \}| \leq 1.$$
\end{definition}

That is, for any given cell, there is at most one coface with a smaller or equal value and at most one face with a larger or equal value. It is possible that both of the above sets have zero cardinality for some $\sigma \in K_p$, in which case we say the cell $\sigma$ is a \textit{critical cell} of \textit{index $p$}. In other words, we say $f$ is critical at a cell $\sigma$ if and only if, locally at $\sigma$, the function $f$ is increasing with the dimension of the cells. We denote by $c_i(f)$ the number of critical cells of index $i$ for a discrete Morse function $f$. 

\begin{theorem}[Forman 2002] \label{morseineq}
Let $f$ be a discrete Morse function on a $n$-dimensional simplicial complex $K$. We then have that\footnote{Yes, the second statement is indeed called an ``inequality" in the literature.}
\begin{enumerate}
    \item  \textit{First Weak Morse Inequality:} $\beta_i(K) \leq c_i(f)$ for all $i$, and
    \item  \textit{Second Weak Morse Inequality:} \[\sum_{i=0}^n (-1)^i c_i(f) = \chi(K).\]
\end{enumerate}
where $\chi(K)$ denotes the Euler characteristic of $K$, i.e. $\sum_{i=0}^n (-1)^i|K_i|$.
\end{theorem}

Note that there always exists a discrete Morse function on a complex, e.g.\ the trivial function $f:K \to \RR$ defined by
$$f(\sigma) = \dim \sigma,$$
where $\dim \sigma$ is the dimension of the cell $\sigma$, is a discrete Morse function on $X$. Under this function, every cell is critical. The constant function is not a Morse function unless one is merely considering disjoint unions of $0$-cells.

Historically, discrete Morse functions were studied on regular CW complexes, though today research focuses on the case of simplicial complexes. This is largely due to their equivalence: a regular CW complex is \textit{isomorphic} to a simplicial complex in a natural definition of the category of combinatorial cell complexes. For an explicit discussion of this, we direct the reader to \cite{cwsimplicial}.

We can now build towards discrete gradient vector fields. The following lemma will make sure the definition we provide will be well-defined.

\begin{lemma} \label{exclusion} \label{unique}
Suppose $\sigma$ is a $p$-cell that is not critical for a discrete Morse function $f$. Then exactly one of the following conditions holds:
\begin{enumerate}
    \item There is a $(p+1)$-cell $\tau > \sigma$ with $f(\tau)\leq f(\sigma)$;
    \item There is a $(p-1)$-cell $\alpha < \sigma$ with $f(\alpha)\geq f(\sigma)$.
\end{enumerate}
\end{lemma}

\begin{proof}
Assume for contradiction both conditions hold for $\sigma$. We then have that $\tau$ has another $(p-1)$-cell, say $\sigma'$, where $\alpha<\sigma'<\tau$. Since $\alpha<\sigma$ and $f(\alpha) \geq f(\sigma)$, we have that $f(\alpha) < f(\sigma')$. Similarly, $f(\sigma') < f(\tau)$. Hence we get that 
$$f(\tau) \leq f(\sigma) \leq f(\alpha) < f(\sigma') < f(\tau),$$
which is a contradiction.
\end{proof}

\

\begin{center}
    \captionsetup{type=figure}
    \includegraphics[width=.35\linewidth]{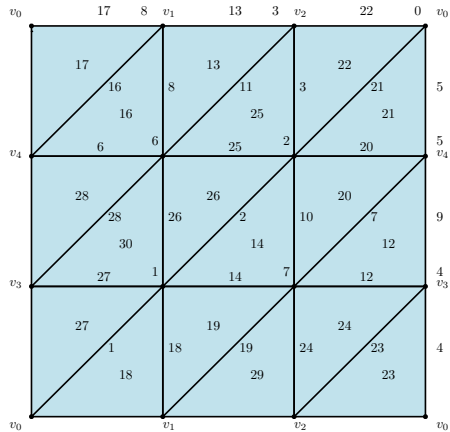} \label{torus}
    \captionof{figure}{An example of a discrete Morse function on a triangulation of a torus.} 
\end{center}

\

\begin{definition} \label{gendef}
A \textit{discrete vector field} $V$ on a simplicial complex $K$ is a collection of pairs of faces where each pair $(\sigma,\tau) \in V$ satisfies:
\begin{enumerate}
    \item $\dim \tau = \dim \sigma +1$,
    \item $\sigma < \tau$, and
    \item every face of the complex is in at most one pair.
\end{enumerate}
If $f$ is a discrete Morse function on $K$, it induces the discrete vector field
$$V_f := \{(\sigma,\tau) \in K \times K \ | \ \dim \tau = \dim \sigma + 1 \text{ and } f(\tau) \leq f(\sigma)\}.$$
We say that a discrete vector field $V$ is a discrete \textit{gradient} vector field if it is induced by a discrete Morse function on $K$. If $(\sigma,\tau) \in V$, we call $(\sigma,\tau)$ an \textit{arrow}, \textit{matching}, or \textit{vector}, with $\tau$ the \textit{head} and $\sigma$ the \textit{tail}.
\end{definition}

Observe that fitting orientations are a special case of discrete vector fields, where pairings occur only in the two top-most dimensions. In general, fitting orientations do not coincide with discrete gradient vector fields, as the latter is in-fact characterized by not having \textit{closed paths of faces}.

\

\begin{center}
    \captionsetup{type=figure}
    \includegraphics[width=.8\linewidth]{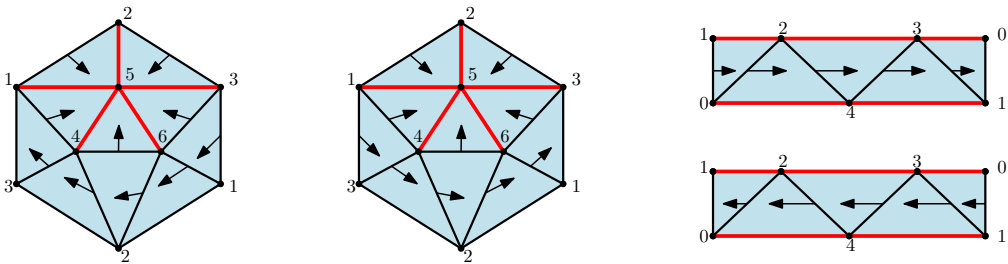} 
    \captionof{figure}{Examples of fitting orientations having closed cycles on rooted forests of  triangulations of the real projective plane and the M\"{o}bius strip.} 
    \label{proj_mobius}
\end{center}

\

We conclude with presenting an important equivalent way of viewing discrete gradient vector fields. Recall that given a simplicial complex $K$, we may associate to it a \textit{Hasse diagram}\footnote{We sometimes abuse notation and treat the Hasse diagram $\mathcal{H}_K$ alternately as both a poset and a graph. It will be clear from context.}, which we denoted by $\mathcal{H}_K$, defined as follows: if $\sigma \in K$, we write $\sigma \in \mathcal{H}_K$ for the corresponding node. There is an edge between two vertices $\sigma,\tau \in \mathcal{H}_K$ if and only if $\tau$ is a face of co-dimension-$1$ of $\sigma$. It is standard to draw Hasse diagrams with levels increasing from the base (see Figure $5$).

Recall that a \textit{matching} on a graph $\Gamma$ is a subset $M \subseteq E_\Gamma$ of the edges such that no two edges in $M$ have common vertices. Importantly observe that if $V$ a discrete vector field on $K$, then $M_V:=V$ is a(n edge) matching on $\mathcal{H}_K$. If $V_f$ is a discrete gradient vector field, then $M_f := M_{V_f}$ is  called an  \textit{acyclic} matching on $\mathcal{H}_K$. In this language, fitting orientations on a rooted forest $(F,R)$ may be seen as matchings on the subposet $\mathcal{H}_{(F,R)}$ of $\mathcal{H}_K$ induced by $F \cup \bar{R}$.

\begin{remark}
Acyclic matchings on $\mathcal{H}_K$ are also sometimes called \textit{Morse matchings} and the computational complexiy of finding them is something of a great deal of interest, as it allows us to find a smaller complex which is homotopy equivalent to the original complex. For more on this, \cite{complexity} is a great place to start.
\end{remark}

A matching $M_V$ on $\mathcal{H}_K$ from a discrete vector field $V$ are commonly illustrated by orienting edges on $\mathcal{H}_K$: if $e \in M_V$, it is oriented upwards, otherwise it is oriented downwards. The matching is acyclic if and only if the corresponding oriented Hasse diagram has no directed cycles\footnote{In fact, more is known: if an oriented Hasse diagram arising from a matching contains a cycle, it is \textit{necessarily} contained within two levels. See \S2.2 of \cite{scoville} for details.}.


\

\begin{center}
    \captionsetup{type=figure}
    \includegraphics[width=.5\linewidth]{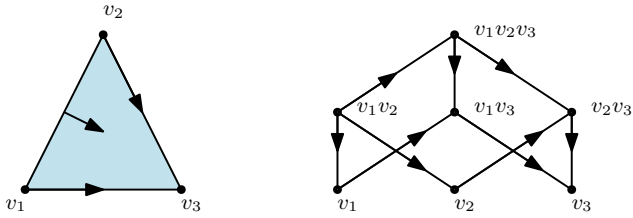} \label{hasse_ex}
    \captionof{figure}{An example of a discrete gradient vector field $V$ on the $2$-simplex $K = \Delta^2$ and the associated orientation on the Hasse diagram $\mathcal{H}_K$.} 
\end{center}

\

We conclude by mentioning the specific sets we aim to enumerate.

\begin{definition}
For a simplicial complex $K$, let $\mathcal{M}(K)$ denote the set of equivalence classes of discrete Morse functions on $K$, where $f \sim g$ if and only if $V_f = V_g$. Furthermore, for any $\ell \in \ZZ_{\geq 0}$, define the subset $\mathcal{M}_\ell(K) \subseteq \mathcal{M}(K)$ by
$$\MM_\ell(K) := \{ f \in \mathcal{M}(K) \ | \ c_d(f) = \ell \text{ and } c_n(f) = |K_n| \text{ for } n < d-1 \}.$$
\end{definition}

 Observe it follows from the Second Weak Morse Inequality that as we have $c_d(f) = \ell$ and $c_n(f) = |K_n| \text{ for } n < d-1$, then from canceling terms, we obtain
$$c_d(f) - c_{d-1}(f) = |K_d| - |K_{d-1}|,$$
so as $c_d(f) = \ell$ for any $f \in \MM_\ell(K)$, it necessarily follows that $c_{d-1}(f) = \ell - |K_d| + |K_{d-1}|$ for all discrete Morse functions $f \in \MM_\ell(K)$. 

\

An equivalent formulation of this is having $\mathcal{M}(K)$ denote the set of \textit{acyclic} matchings on $\mathcal{H}_K$ and $\mathcal{M}_\ell(K)$ denote the set of \textit{acyclic} matchings on the subposets $\mathcal{H}_{(F,R)}$ such that 
$$|F|=|K_d|-\ell \qquad \text{and} \qquad |\bar{R}| = \ell - |K_d| + |K_{d-1}|.$$

\section{Gradients of \texorpdfstring{$1$} --dimensional Simplicial Complexes}

In this section, we will discuss our results towards enumerating all the discrete gradient vector fields that may be imposed on a $1$-dimensional complex, i.e.\ a graph. We approach this by finding a \textit{generating polynomial}, and furthermore, we show that this generating polynomial is equal to the characteristic polynomial of the graph Laplacian.

Throughout, assume $\Gamma$ is a finite, simple, and not neccesarily connected graph and $f$ is a discrete Morse function on $\Gamma$. We commonly denote a subgraph $\Gamma'$ of $\Gamma$ with the notation $\Gamma' \subseteq \Gamma$, and write $f|_{\Gamma'}$ for the discrete Morse function induced on the subgraph by $f$. All graphs are considered labeled, so their automorphism group is always trivial.



We will identify the discrete gradient vector field $V_f$ with a directed graph $\Gamma_f$ in the following natural way:\ the underlying undirected graph of $\Gamma_f$ is $\Gamma$ and an edge $(v,w)$ is directed in $\Gamma_f$ if and only if $(\{v\},\{v,w\}) \in V_f$, otherwise it is considered undirected. This provides a helpful characterization of critical cells in the $1$-dimensional case:
\begin{itemize}
    \item an edge is a critical cell of $f$ on $\Gamma$ if and only if it is undirected in $\Gamma_f$,
    \item a vertex is a critical cell of $f$ on $\Gamma$ if and only if it is a sink in $\Gamma_f$.
\end{itemize}

\subsection{Characterizing gradients}

We now can begin descending towards the heart of the enumeration problem. The counting of gradients on a graph becomes clear once we begin removing critical edges, which is precisely what the next definition gives us the ability to focus more closely on.

\begin{definition}
Let $f$ be a discrete Morse function on the graph $\Gamma$ and $X \subseteq E_\Gamma$ be the set of critical edges under $f$. Then the graph $R(\Gamma,f)$, which we call the \textit{remainder}, is the edge-subgraph of $E_{\Gamma_f} \setminus X$ in the directed graph $\Gamma_f$.
\end{definition}

We can now construct a natural bijection between discrete gradient vector fields on graphs and rooted forests. Note that this was obtained through different means in \cite{chari}, relying instead on the acyclicity property of gradients.

\begin{lemma}
For a graph $\Gamma$, we have a bijection
$$\mathcal{M}(\Gamma) \cong \{ F \text{ rooted forest} \ | \ F^* \subseteq \Gamma \},$$
where $F^*$ denotes the underlying undirected graph of $F$, given by the map $f \mapsto R(\Gamma,f)$.
\end{lemma}

\begin{proof}
Consider some $f \in \MM(\Gamma)$. Let $X \subseteq E_\Gamma$ denote the set of critical edges on $\Gamma$ under $f$. If $Z \subseteq \Gamma$ is a cycle, then by the second weak Morse inequality,
$$c_1(f|_Z) \geq 1 = \beta_1(Z),$$
hence the edge-subgraph of $E_\Gamma\setminus X$ in $\Gamma$, say $\Lambda$, must be a forest. To see that $\Lambda$ has a distinguished root for each component, let $T$ be a connected component. Then by the first weak Morse inequality,
$$c_0(f|_T)+c_1(f|_T)=1+0=\chi(T)=\beta_0(T),$$
where the second equality is due to the construction of $\Lambda$. Thus it follows that each connected component has a sink, therefore $\Lambda$ corresponds to some $R(\Gamma,f)$.

We now construct the inverse map. Without loss of generality\footnote{We repeat the same process on each component for forests having more than one connected component.}, assume $F=T$ is tree rooted at a vertex $v$. Denote by $\widetilde{T}$ the graph resulting from adding the edges $X := E_\Gamma \setminus E_T$ to $T$ and associated additional vertices. We define a discrete Morse function $g$ on on $\widetilde{T}$ as follows: define $g(v) = 0$, and increment the assignment to cells by $1$ while traversing outwards along edges from the root (``against the grain"). For any edges $e \in X$, if $v,w$ are the endpoints of $e$, define $g(e)=\max\{g(v),g(w)\}+1$. The resulting assignment defines a discrete Morse function on $\widetilde{T}$ that trivially satisfies $R(\widetilde{T},g)=T$. Since we clearly have $\widetilde{T}=\Gamma$, the proof is complete.
\end{proof}


 In particular, this above map can be refined to a collection of the following maps.

\begin{corollary} \label{rooted1dim}
For a graph $\Gamma$, we have a bijections
$$\mathcal{M}_k(\Gamma) \cong \{ F \text{ rooted forest} \ | \ F^* \subseteq \Gamma \text{ and } |E_F| = |E_\Gamma|-k \},$$
for each $0 \leq k \leq |E|$, given by the map $f \mapsto R(\Gamma,f)$.
\end{corollary}

It is worth noting that by the first weak Morse inequality, we necessarily have that $\MM_k(\Gamma)$ can be non-trivial only when $\beta_1(\Gamma) \leq k \leq |E|$. We highlight the interesting minimal case.

\begin{corollary} \label{eigen}
We have that for a graph $\Gamma$ on $n$ vertices that
\begin{equation*} \label{eigeneqn} \tag{*}
|\MM_{\beta_1}(\Gamma)| = \lambda_1 \cdots \lambda_{n},
\end{equation*}
where the $\lambda_i$ are the non-zero eigenvalues of the graph Laplacian of $\Gamma$.
\end{corollary}

\begin{proof}
Observe that
\begin{align*}
    |\MM_{\beta_1}(\Gamma)|
    &= |\{ F \text{ rooted forest} \ | \ F^* \subseteq \Gamma \text{ and } |E_F| = |E_\Gamma|-\beta_1(\Gamma) \}| \\
    &= |\{ T \text{ tree} \ | \ \text{$T$ rooted spanning tree of $\Gamma$} \}|\\
    &= |V_\Gamma| \cdot |\{ T \ | \ \text{$T$ spanning tree of $\Gamma$}\}|\\
    &= |V_{\Gamma}| \cdot \left( \frac{1}{|V_{\Gamma}|} \lambda_1 \cdots \lambda_{n} \right) &\text{by Kirchhoff's Theorem} \\
    &= \lambda_1 \cdots \lambda_{n}
\end{align*}
as desired.
\end{proof}

\begin{remark} \label{morsecomplex}
It is worth drawing a connection between this corollary and the Morse (simplicial) complex, an object of current interest in discrete Morse theory. The Morse complex of $\Gamma$, $\mathfrak{M}(\Gamma)$, is the simplicial complex whose $n$-simplices are precisely the acyclic $n$-matchings of $\mathcal{H}_\Gamma$. (One can of course talk of the Morse complex of a general simplicial complex.) It was shown \cite[Proposition 10]{ayala} that the the number of simplicies of maximal dimension in $\mathfrak{M}(\Gamma)$, for a graph $\Gamma$ on $n$ vertices, is $($\ref{maineq}$)$ as above. Corollary \ref{eigen} proves this statement, as these cells are precisely in correspondence with the elements of $\MM_{\beta_1}(\Gamma)$. For more on the Morse complex, see \cite[Chapter 7]{scoville}.
\end{remark}

\begin{center}
    \captionsetup{type=figure}
    \includegraphics[width=.35\linewidth]{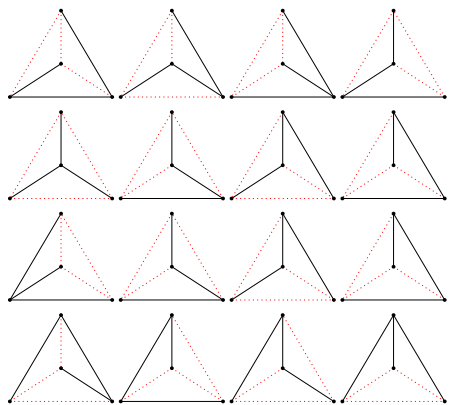}
    \captionof{figure}{All the possible 16 minimal critical edge configurations on the wheel graph $W_4$. Note that the black edges represent the remainders, hence every graph in the above figure represents $4$ different gradients; one for every different vertex considered as a root.}
\end{center}

\

Observe that when $f$ is a discrete Morse function on a graph $\Gamma$ where $c_1(f)=\beta_1 + k$ for some $0\leq k \leq |E|-\beta_1(\Gamma)$, we have that the number of forests is given by
$$\beta_0(R(\Gamma,f)) = c_0(f|_{R(\Gamma,f)}) + c_1(f|_{R(\Gamma,f)}) + \beta_1(R(\Gamma,f)) = c_0(f|_{R(\Gamma,f)}),$$
where the first equality is due to the second weak Morse inequality. Note that, in general, $c_0(f|_{R(\Gamma,f)}) \leq c_0(f)$, as the LHS counts the number of critical vertices that are \textit{not} embedded into the edges that are critical.

We can now present the main result in the $1$-dimensional case.

\begin{theorem} \label{charpoly}
If $\Gamma$ is a graph, then
$$\det(\Delta + \lambda I) = \sum_{i=0}^{n} |\MM_{m-i} (\Gamma)| \lambda^{n-i},$$
where $n = |V_\Gamma|$ and $m = |E_\Gamma|$.
\end{theorem}

\begin{proof}
We have by Corollary \ref{rooted1dim} that all the elements of $\MM_{\beta_1+k}(\Gamma)$ are edge-subgraphs of $\Gamma$ with $(n-k-1)$-many edges and are forests. Each forest $F=R(\Gamma,f)$ for $f \in \MM_{\beta_1+k}(\Gamma)$ is repeated $|V_F|$-many times, one for each choice of a root, as accounted for. Lastly, observe that since we wish for $n-k-1=i$, this implies $k=n-i-1$, and using the formula for the Euler characteristic of a graph to simplify this expression, we may obtain 
$$\beta_1 + k = \beta_1 + n -i - 1 = m-i.$$
After applying Theorem \ref{biggsthm}, the proof is complete.
\end{proof}

It can be notably observed from the above formulation that the minimal (nonzero) term of $\det(\Delta + \lambda I)$ is the coefficient $|\MM_{0}(\Gamma)|$ of $\lambda^{n-m} = \lambda^{\chi(\Gamma)}$, as $\MM_k(\Gamma) = \emptyset$ for $k < 0$.

\begin{corollary}
The number of gradients $|\mathcal{M}(\Gamma)|$ on a graph $\Gamma$ is $\det(\Delta + I)$.
\end{corollary}

\begin{remark}
Theorem \ref{charpoly} has been previously obtained though in a different context: recall the definition of the Morse complex from Remark \ref{morsecomplex}. Chari and Joswig showed in \cite[Corollary 3.3]{chari} that the $f$-vector of $\mathfrak{M}(\Gamma)$ satisfies
$$\det(\Delta + \lambda I) = \sum_{i=0}^{n} f_{m-i} \lambda^{n-i},$$
which is exactly the content of the above theorem, as the component $f_k$ counts the acyclic $k$-matchings, hence corresponds to gradients having $m-k$ critical cells. 

This statement connecting the Laplacian and the $f$-vector of the Morse complex sadly does not hold in higher dimensions: as we will soon see, the coeffecients of the characteristic polynomial of the $d$th Laplacian $\Delta_d$ for a $d$-dimensional simplicial complex $K$ do not correspond with acyclic matchings on all of $\mathcal{H}_K$, but rather just the two top-most levels.
\end{remark}

\begin{example}
Consider the cycle $C_4$. The graph Laplacian matrix is
$$\Delta = \begin{pmatrix} 2 & -1 & 0 & -1 \\ -1 & 2 & -1 & 0 \\ 0 & -1 & 2 & -1 \\ -1 & 0 & -1 & 2 \end{pmatrix}$$
and so the characteristic polynomial is $\lambda^4 + 8\lambda^3+20\lambda^2+16\lambda$. Observe the below figure for the collection of all gradients, illustrating the content of the above theorem.

\

\begin{center}
    \captionsetup{type=figure}
    \includegraphics[width=.8\linewidth]{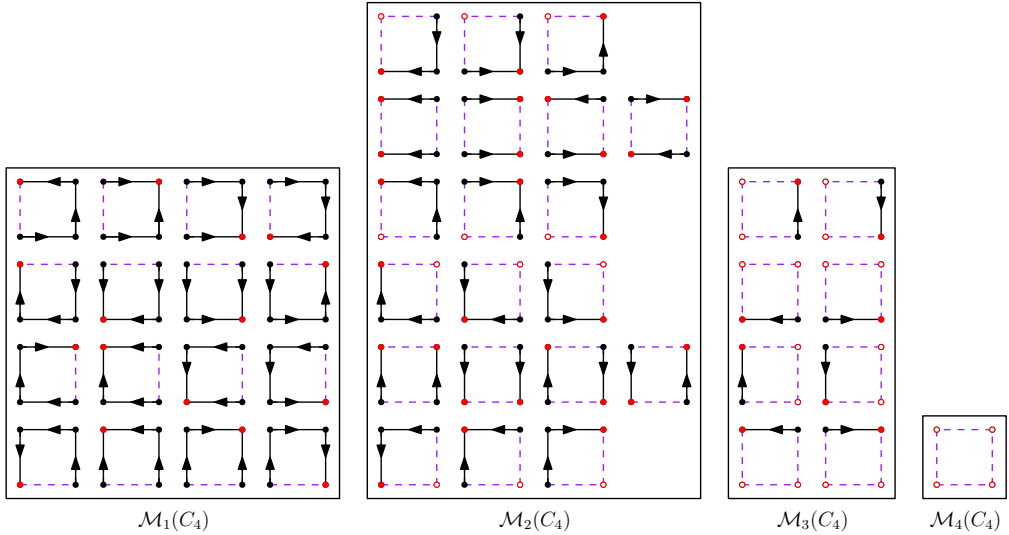}
    \captionof{figure}{The subsets $\MM_k(C_4)$ for $\beta_1(C_4) \leq k \leq |E_{C_4}|$. The dashed purple edges denote critical edges, the filled red circles denote critical vertices that are present in the the remainder, and the empty red circles denote the other critical vertices which are not present in the remainder.}
\end{center}

\end{example}

\begin{example}
Consider the star graph on $n+1$ vertices, $S_n$. Since $\beta_1(S_n)=0$, we are not ``obligated" to put any critical edges anywhere, we merely choose any edge, if any, to be critical. In the case of $|\MM_k(S_n)|$ for instance: we have all $\binom{n}{k}$-many choices for a critical edge. For every choice, the remainder requires a root, thus there are $(n+1-k)$-many vertices left to choose from. Therefore
$$|\MM_k(S_n)| = \binom{n}{k} (n+1-k),$$
and so by Theorem \ref{charpoly}, the characteristic polynomial of the Laplacian of $S_n$ is given by
$$\det(\Delta + \lambda I) = \sum_{i=0}^{n+1} \binom{n}{i}(i+1)\lambda^{n-i+1}.$$
\end{example}

\subsection{Matching polynomials and the adjacency matrix}

A simple, yet surprising connection between the graph Laplacian's characteristic polynomial and the characteristic polynomial of the adjacency matrix results from Theorem \ref{charpoly}. Recall that for a graph $\Gamma$, an $r$-matching is a matching having $r$-many elements. Let us denote the number of $r$-matchings on a graph $\Gamma$ by $m_r$.

\begin{definition}
The \textit{matching polynomial} of a graph $\Gamma$ with $n$ edges is given by
$$m_{\Gamma} (x) = \sum_{k \geq 0} (-1)^k m_k x^{n-2k}.$$
\end{definition}

As corollary to Theorem \ref{charpoly}, we can cross a bridge built in \cite{matching-adj}, proving that the matching polynomial and the characteristic polynomial of the adjacency matrix are equal if and only if the given graph is a forest. As gradients on a forest $F$ are equivalent to acyclic matchings on the Hasse diagram $\mathcal{H}_F$, which is certainly a forest as it is isomorphic to the subdivision of $F$, we obtain the following. Let $A_\Gamma$ and $\Delta_\Gamma$ denote the adjacency matrix and the Laplacian matrix of a graph $\Gamma$, respectively.

%
%

\begin{theorem}\label{forestthm}
Let $F$ be a forest graph on $n$-many vertices having $m$-many edges, and write $\det(\Delta_F+\lambda I) = \sum_{i = \beta_0}^n a_i \lambda^i$. Then the characteristic polynomial of the adjacency matrix of $\mathcal{H}_F$ is 
$$\det(A_{\mathcal{H}_F}+\lambda I) = \sum_{k=0}^m a_{n-k} \lambda^{n-2k}.$$
\end{theorem}

\begin{proof}
It follows from the above discussion that by \cite{matching-adj}, $m_{\mathcal{H}_F}(\lambda) = \det(A_{\mathcal{H}_F} - \lambda I)$. Because $\mathcal{H}_F$ is a forest, all matchings are acyclic, so by Theorem \ref{charpoly}, $m_k = |\MM_{m-k}(\Gamma)| = a_{n-k}$. 
\end{proof}

\section{Gradients of General Simplicial Complexes}

Consider now some discrete Morse function $f \in \MM(K)$ where $K$ is a $d$-dimensional simplicial complex. We now associate a pair $(F,R)$, where $F$ is the set\footnote{To save time and avoid unnecessary over-precision, we sometimes use $F$ to also denote the \textit{complex generated} by the set of non-critical $d$-cells of $K$. Respectively with the root $R$.} of non-critical $d$-cells of $K$, and $R$ denotes the set of critical $(d-1)$-cells of $F \subseteq K$. As one expects from the notation, we will show that $(F,R)$ is a rooted forest of $K$, and this is indeed only dependent on the gradient, so the definition is well-defined.

\subsection{Characterizing gradients}





\begin{theorem} \label{rootedfor}
If $f \in \MM_\ell(K)$, then the pair $(F,R)$ is a rooted forest of $K$.
\end{theorem}

\begin{proof}
Since $F$ is a subcomplex of $K$, $f|_F$ is a discrete Morse function on ${F}$, and by definition of $F$ we have $c_d\left( f|_{F} \right) = 0$. As $H_d(F) = \ker \partial_d^{F}$ since $F_{d+1} = \emptyset$, and because $c_d(f|_F) \geq \beta_d(F) = |H_{d}(F)|$ by the first weak Morse inequality, we obtain that $\ker \partial_d^{F} = 0$ so $F$ is a forest. 

\

We now show that $R$ is a root of ${F}$. If $F_d = \emptyset$, then $R = \emptyset$, and we are trivially done. Assume moving forward that $F_d$ is non-empty.

\

\textit{$R$ is relatively-free:} Observe that by Lemma \ref{exclusion}, we have that for all $\gamma \in F_d$, there exists a unique $\alpha \in \bar{R} = {F}_{d-1} \setminus R$ such that $\alpha < \gamma$. Hence $\alpha \in \supp \  \partial_d^F(\gamma)$. Lastly, as $C_{d-1}(F) \cong C_{d-1}(R) \oplus C_{d-1}(\bar{R})$, this implies $\im \partial_d^F \cap C_{d-1}(R) = \{0\}$, as every image of a non-zero element of $C_{d}(F)$ must have a non-trivial $C_{d-1}(\bar{R})$ component.

\

\textit{$R$ is relatively-generating:} Consider $\bar{R}$ and note that it is non-empty since $c_i(f)=|K_i|$, and since $c_d(f) = |K_d| - |F_d|$, we must necessarily have that $c_{d-1}(f) = |K_{d-1}| - |F_d|$ by the weak Morse inequality. Now let $\beta \in \bar{R} $. We then have that $\beta$ is not critical under $f$ on $K$ (and ${F}$). Since $c_{d-2}(f) = |K_{d-2}|$, by Lemma \ref{exclusion} it follows that for all $\tau \in K_{d-2}$ that there is no $(d-1)$-cell $\alpha > \tau$ such that $f(\alpha) \leq f(\tau)$. Yet as $\beta$ is not critical, by Lemma \ref{exclusion} applied to $\beta$, due to the above condition on the elements of $K_{d-2}$, we have that there exists some $d$-cell $\gamma > \beta$ such that $f(\gamma) \leq f(\beta)$.

But then $\gamma$ is not a critical $d$-cell by definition, hence it must be an element of $F_d$ and a cell of $F$. Since $\dim F = d$, we have that $(\beta,\gamma)$ is an arrow of the gradient induced by $f$ on $K$ (hence $F$). As $\beta$ is the only non-critical $(d-1)$-cell on the boundary of $\gamma$, we must have that the remaining boundary cells are critical. Therefore, if we write $\partial_d^{F}(\gamma)$ as in $(*)$, then there exists some $0 \leq j \leq d$ where $\gamma|_{[v_0,\dots,\widehat{v}_j,\dots,v_d]} = \beta$ and
$$\beta + \sum_{\begin{smallmatrix} 0 \leq i \leq d \\ i \neq j  \end{smallmatrix}} \gamma|_{[v_0,\dots,\widehat{v}_i,\dots,v_d]}  = \partial_d^{F}(\gamma) \in \im \partial_d^{F},$$
where $\beta \in \bar{R}$ and $\sum_{\begin{smallmatrix} 0 \leq i \leq d \\ i \neq j  \end{smallmatrix}} \gamma|_{[v_0,\dots,\widehat{v}_i,\dots,v_d]} \in C_d(R)$, as needed.
\end{proof}

Note that the above result indeed implies much of our $1$-dimensional work, as if $f$ is a discrete Morse function on a graph $\Gamma$, then we have that $(F,R)$ is precisely the data of $R(\Gamma,f)$. There is a topological viewpoint to these rooted forests as well: we necessarily have that $F \cong R$, that is, $F$ is (simplicially) homotopy equivalent to $R$. This theorem, along with Corollary \ref{conj}, answer two of the open questions posed by Mukherjee in \cite{ToRF} simultaneously, namely, classifying all simplicial forests which are homotopy equivalent to their roots and classifying all rooted forests satisfying \cite[Lemma 3.5]{ToRF}. 

\begin{theorem} \label{collapse}
Let $(F,R)$ be a rooted forest of a $d$-dimensional simplicial complex $K$. Then $(F,R)$ corresponds to a gradient if and only if $F$ simplicial collapses to $R$.
\end{theorem}

\begin{proof}
For the forward direction, this follows from one of the main results of Discrete Morse Theory, as stated by Forman in \cite[Theorem 3.3]{4cellcomplexes}, where we may collapse along the gradient and preserve the homotopy type of our complex.

Since $F$ simplicial collapses to $R$, there must be a sequence of elementary collapses
$$F=F_0 \searrow F_1 \searrow F_2 \searrow \cdots \searrow F_{n-1} \searrow F_n = R,$$
i.e. $F_{i+1} = F_{i}-\{\sigma,\tau\}$ where $(\sigma,\tau) \in (F_i)_{d-1}\times (F_i)_d$ are a pair of simplicies in $F_i$ such that $\sigma$ is a face of $\tau$ and $\sigma$ has no other cofaces. Reversing this, we instead obtain a sequence of elementary expansions
\begin{equation}\label{expand}\tag{$\dagger$}
R = R_0 \nearrow R_1 \nearrow R_2 \nearrow \cdots \nearrow R_{n-1} \nearrow R_{n} = F
\end{equation}
with $R_{i+1} = R_{i} \cup \{\sigma,\tau\}$ with $(\sigma,\tau) \in (R_{i+1})_{d-1}\times (R_{i+1})_d$ a pair of simplices not in $R_i$ where $\sigma$ is a face of $\tau$ and all other faces of $\tau$ are in $R_i$. Using this sequence, we can construct an acylic fitting orientation $\psi: \bar{R} \to F$ on $(F,R)$ as follows: for $0 \leq i \leq n-1$, let $\sigma_i$ and $\tau_i$ denote the cells added to $R_i$ in the elementary expansion to $R_{i+1}$. Since all faces besides $\sigma_i$ of $\tau_i$ are in $R_i$, $\sigma_i$ must not be in the root $R$. Therefore let $\psi(\sigma_i)=\tau_i$. As $R$ ultimately expands to $F$, and at each elementary expansion $R_i \nearrow R_{i+1}$ we added a $d$-cell from $F$ and a $(d-1)$-cell from $\bar{R}$ to $R_i$, this implies $\psi$ is a bijection, as needed.

We now show that $\psi$ is acyclic, i.e. the matching 
$$M_\psi = \{ (\sigma, \tau) \in F_{d-1} \times F_d \ | \ \sigma \in \bar{R} \text{ and } \psi(\sigma) = \tau \}$$
induced by $\psi$ on $\mathcal{H}_K$ is acyclic. Assume for contradiction that this is not the case. We then have a directed sequence
$$\sigma_{k_0}, \ \tau_{k_0}, \ \sigma_{k_1}, \ \tau_{k_1}, \ \dots, \ \sigma_{k_{m}}, \ \tau_{k_{m}},$$
for some $2 \leq m\leq n$, where $\psi(\sigma_{k_i}) = \tau_{k_i}$ and $\sigma_{k_{i+1}}<\tau_{k_i}$ for $0 \leq i \leq m$ and $\sigma_{k_0}<\tau_{k_{m}}$. Consider $\{\sigma_{k_m},\tau_{k_m}\}$, and note by construction of $\psi$, there exists some $0 \leq j \leq n$ in the sequence (\ref{expand}) such that $R_j = R_{j-1} \cup \{\sigma_{k_m},\tau_{k_m}\}$. Since $\sigma_{k_{m}}<\tau_{k_m-1}$ and $\tau_{k_{m-1}} = \psi(\sigma_{k_{m-1}})$, we must have that $R_{j+1} = R_{j} \cup \{\sigma_{k_{m-1}}, \tau_{k_{m-1}}\}$. Continuing this process, we obtain the subsequence of (\ref{expand}) of elementary expansions
$$R_j \nearrow R_{j+1} \nearrow R_{j+2} \nearrow \cdots \nearrow R_{j+(m-1)} \nearrow R_{j+m},$$
where $R_{j+i} = R_{j+(i-1)} \cup \{\sigma_{k_{m-i}},\tau_{k_{m-i}}\}$ for $0 \leq i \leq m$. But we have that $\sigma_{k_0} < \tau_{k_m}$, so we also must have that $R_{j+m+1} = R_{j+m} \cup \{\sigma_{k_m},\tau_{k_m}\}$. Yet this is not a valid elementary expansion, as these are simplicies that are already in $R_{j+m}$. Therefore this directed sequence cannot exist, hence $M_{\psi}$ is acyclic, and so $\psi$ corresponds to a gradient.
\end{proof}

The backwards direction in the above proof also shows that if $(F,R)$ corresponds to a gradient, then that gradient is necessarily unique, since there was only one possible fitting orientation we could construct. We prove this differently and explicitly in Lemma \ref{last}.

\

\begin{center}
\captionsetup{type=figure}
\includegraphics[width=1\linewidth]{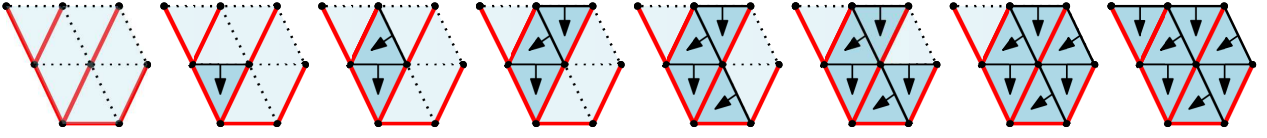}
\captionof{figure}{An example of a sequence of elementary expansions from a root $R$ to a forest $F$ giving rise to a discrete gradient vector field on the simplicial complex generated by $F$.}
\end{center}

\begin{corollary} \label{conj}
With the same notation as above, if $(F,R)$ is a rooted forest that corresponds to a discrete gradient vector field, then $H_{d-1}(F,R) \cong 0$.
\end{corollary}

\begin{proof}
We have that $H_{d-1}(F,R) \cong H_{d-1}(F/R)$ by \cite[Theorem 2.27]{hatcher} and \cite[Proposition 2.22]{hatcher}. By Theorem \ref{collapse}, $F$ is homotopy equivalent to $R$, hence $F/R$ is a contractible topological space. Therefore $H_{d-1}(F/R) \cong 0$.
\end{proof}

Note the converse of Corollary \ref{conj} does not hold in general (see Example \ref{bipyramid} below). It was proven by Mukherjee in \cite[Theorem 4.2]{ORF} that the converse does indeed hold when the simplicial complex is a triangulation of an orientable $d$-manifold.




Observe now that we may rewrite the sum of Theorem \ref{higherlap} as
\begin{equation*} \tag{$*$}
    \det(\Delta_d + \lambda I) = \sum_{i=0}^n  \sum_{\begin{smallmatrix} \text{$(F,R)$ rooted forest} \\ \text{with $|R|=i$} \end{smallmatrix}} |H_{d-1}(F,R)|^2 \lambda^{i}
\end{equation*}
where $n = |K_{d-1}|$. We now restrict our view to when $K$ is a triangulation of an orientable $d$-manifold, where we are aided by the work of Mukherjee, who closely studied the topology of higher dimensional rooted forests in this case. Mukherjee showed in \cite[Theorem 4.5]{ORF} that for all rooted forests $(F,R)$ of $K$, $F$ simplicial collapses to $R$, hence by Corollary \ref{conj}, we may take all the homological weight terms to be $1$. \rmat Therefore $(*)$ reduces to $\sum_{i=1}^n q_i \lambda^{n-i}$, where $q_i$ is the number of rooted forests $(F,R)$ of $K$ with $|R|=i$.

What then remains to be shown is a bijection between $\MM_{m-i}(K)$ and rooted forests of $K$ having $i$-many $(d-1)$-cells in the root. Due to \cite[Corollary 3.7]{ToRF}, we have that this map is surjective when $K$ is a triangulation of an orientable $d$-manifold. In particular, it is shown that there exist at least $|H_{d-1}(F,R)|$-many gradients giving rise to the rooted forest $(F,R)$. We show that this map is injective by proving for an arbitrary simplicial complex $K$, if a rooted forest comes $(F,R)$ from a gradient $\phi$, there cannot exist a distinct gradient $\widetilde{\phi}$ giving rise to the same rooted forest.

\begin{lemma} \label{last}
If $K$ is a simplicial complex, then the map
$$\MM_{m-i}(K) \ \rightarrow \  \left\{ (F,R) \text{ rooted forest of $K$} \ | \ |R|=i \right\},$$
with $m = |K_d|$ and $1 \leq i \leq |K_{d-1}|$, is injective.
\end{lemma}

\begin{proof}
Begin by assuming for contradiction that it does not hold. Then there exists two discrete Morse functions $\phi, \widetilde{\phi} \in \MM_{m-i}(K)$ such that they induce the same rooted forest $(F,R)$. By the definition of rooted forests induced from discrete Morse functions, we have that $\phi$ and $\widetilde{\phi}$ have the same critical cells in dimensions $d$ and $d-1$. In other words, the corresponding subsets $M_\phi, M_{\widetilde{\phi}} \subseteq E(\mathcal{H}_{K})$ acyclically match the subsets $F \subseteq K_d$ and $\bar{R} \in K_{d-1}$ two different ways. We claim this is impossible.

Let $\bar{R} = \{\sigma_1, \dots, \sigma_k\}$, $F = \{\tau_1, \dots, \tau_k\}$ for some $k \in \NN$, and have $\psi, \widetilde{\psi}$ denote the fitting orientations corresponding to the gradients of $\phi, \widetilde{\phi}$, respectively. Assume without loss of generality that $\psi(\sigma_j) = \tau_{j}$ for $1 \leq j \leq k$. It then follows that $\widetilde{\psi}(\sigma_j)=\tau_{{\pi(j)}}$ for some permutation $\pi \in S_k$. Note that $\pi \neq e$, as otherwise $\phi = \widetilde{\phi}$. We may then rewrite $\pi$ as a product of disjoint cycles $\gamma_1 \gamma_2 \cdots \gamma_n$, and consider a cycle having length greater than $1$, say (without loss of generality) $\gamma_1 = (\alpha_1 \ \alpha_2 \ \alpha_3 \cdots \alpha_m)$. Thus we have that
$$\widetilde{\psi}(\sigma_{\alpha_1}) = \tau_{\alpha_2}, \  \widetilde{\psi}(\sigma_{\alpha_2}) = \tau_{\alpha_3}, \  \dots, \  \widetilde{\psi}(\sigma_{\alpha_{m-1}}) = \tau_{\alpha_m}, \  \widetilde{\psi}(\sigma_{\alpha_m}) = \tau_{\alpha_1}.$$
But note that we also have that $\sigma_{\alpha_j} < \tau_{\alpha_j}$ for $1 \leq j \leq m$ under the poset relation in $\mathcal{H}_K$. Therefore there is in-fact a directed cycle in the fitting orientation $\widetilde{\psi}$:
$$  \sigma_{\alpha_1}, \ \widetilde{\psi}(\sigma_{\alpha_1}) = \tau_{\alpha_2}, \ \sigma_{\alpha_2}, \widetilde{\psi}(\sigma_{\alpha_2}) = \tau_{\alpha_3}, \  \dots, \ \sigma_{\alpha_m}, \ \widetilde{\psi}(\sigma_{\alpha_m}) = \tau_{\alpha_1}, \ \sigma_{\alpha_1}.$$
But this contradicts that $M_{\widetilde{\phi}}$ is an acyclic matching, therefore the map must be injective.
\end{proof}

With the above work and Lemma \ref{last}, we have proven the following result.

\begin{theorem} \label{main}
If $K$ is a triangulation of an orientable $d$-manifold, then
\begin{equation*} \tag{$\dagger$} \label{maineq}
    det( \Delta_d + \lambda I) = \sum_{i=0}^{n} |\mathcal{M}_{m-i} (K)| \lambda^{n-i},
\end{equation*}
where $m = |K_d|$ and $n = |K_{d-1}|$.
\end{theorem}

Note that we've also thus shown the following more general statement as well.

\begin{theorem} \label{main2}
If $K$ is a simplicial complex such that for every rooted forest $(F,R)$, the forest $F$ simplicial collapses to the root $R$, then equation $($\ref{maineq}$)$ is satisfied.\rmat
\end{theorem}

We also have a good understanding of how the gradient generating function in $($\ref{maineq}$)$ behaves when it is not equal to $\det( \Delta_d + \lambda I)$. As the elements of $\mathcal{M}_i(K)$ correspond to a unique rooted forest by Lemma \ref{last} for arbitrary complexes, the main issue is that not every rooted forest may be induced by a gradient, i.e. surjectivity of the map in the lemma. Therefore the below statement follows from our above work and Theorem \ref{higherlap}.

\begin{porism} \label{pathology}
Let $K$ be an arbitrary $n$-dimensional simplicial complex. Then, for $0 \leq i \leq n$, the coefficient of $\lambda^i$ in $\sum_{i=0}^{n} |\mathcal{M}_{m-i} (K)| \lambda^{n-i}$ is less than or equal to the coefficient of $\lambda^i$ in $det( \Delta_d + \lambda I)$. The difference $\varepsilon_i$ between the coefficients is
$$\varepsilon_i = \sum_{\Lambda_i(K)} |H_{d-1}(F,R)|^2$$
where the set $\Lambda_i(K)$ consists of the rooted forests $(F,R)$ of $K$ with $|R|=i$ where $F$ does not simplicial collapse to $R$.
\end{porism}



\begin{corollary}
Let $K$ be a simplicial triangulation of a non-orientable $d$-manifold. Then $\varepsilon_{s}>0$, where $s=|K_{d-1}|-|K_d|$.
\end{corollary}

\begin{proof}
By \cite[Proposition 3.8]{ToRF}, if $K$ is the triangulation of a non-orientable $d$-manifold, we then have that all of $F=K_d$ is a forest that contains torsion in $H_{d-1}(F)$. Let $R$ be a root of $F$. Then $H_{d-1}(R)$ does not contain torsion due to being the top homology group of $R$, so  $F$ cannot simplicial collapse to $R$. Thus by Theorem \ref{collapse}, $(F,R)$ does not correspond to a gradient. By \cite[Lemma 14]{higherdimforests}, $(F,R)$ is a rooted forest a simplicial complex $K$ if and only if $|\bar{R}|=|F|$ and $H_d(F,R) = 0$. Since $F=K_d$, we have that $|R| = |K_{d-1}|-|K_d|$, thus $\varepsilon_s>0$.
\end{proof}

\begin{example} \label{mobfail}
Let $K$ be the simplicial complex that arises from the triangulation of the M\"{o}bius strip that appears in Figure \ref{proj_mobius}. We have that
\begin{align*}
    &det( \Delta_d + \lambda I) \quad &= \quad &\lambda^{10} + 15\lambda^9 + 85\lambda^8 + 225\lambda^7 +275\lambda^6 + 125\lambda^5, \\
    &\sum_{i=0}^{n} |\mathcal{M}_{m-i} (K)| \lambda^{n-i} \quad &= \quad &\lambda^{10} + 15\lambda^9 + 85\lambda^8 + 225\lambda^7 +275\lambda^6 + 121\lambda^5.
\end{align*}
In this case $\Lambda_5(K)$ is non-empty, and has precisely the one rooted forest $(F,R)$: the case where $F=K_2$ are all the $2$-cells and $R$ is the boundary of the strip, as highlighted in red in Figure \ref{proj_mobius}. Hence we obtain that
$$\varepsilon_5 = |H_1(F,R)|^2 = |\ZZ/2\ZZ|^2 = 2^2 = 4.$$
\end{example}

Let us illustrate a strange example where we do not have equality in $($\ref{maineq}$)$. The below example is not the triangulation of any manifold. Though, it is curious as a counterexample, as it is $\mathbb{Z}$-acyclic in positive codimension ($\ZZ$-APC), i.e. $H_i(K;\ZZ)=0$ for all $i<\dim K$, which is a property satisfied by many complexes of combinatorial interest \cite{simptrees}.

\begin{example} \label{bipyramid}
Consider the case where $K$ is the equatorial bipyramid, which is the $2$-dimensional simplicial complex having seven $2$-cells and is the image under $\partial_3$ of the complex generated by $\{0,1,2,3\}$ and $\{1,2,3,4\}$. Here, we have that
\begin{align*}
    &det( \Delta_d + \lambda I) \quad &= \quad &\lambda^9 + 21\lambda^8 + 174\lambda^7+710\lambda^6+1425\lambda^5+1125\lambda^4, \\
    &\sum_{i=0}^{n} |\mathcal{M}_{m-i} (K)| \lambda^{n-i} \quad &= \quad &\lambda^9 + 21\lambda^8 + 174\lambda^7+710\lambda^6+1425\lambda^5+1119\lambda^4.
\end{align*}
\noindent We have that $\Lambda_4(K)$ consists of six rooted forests which are all symmetric to one another, as indicated in Figure \ref{sym} below. As each of these rooted forests $(F,R)$ have trivial $H_1(F,R)$, it follows that $\varepsilon_4 = |\Lambda_4(K)| = 6$.

\

\begin{center}
    \captionsetup{type=figure}
    \includegraphics[width=.7\linewidth]{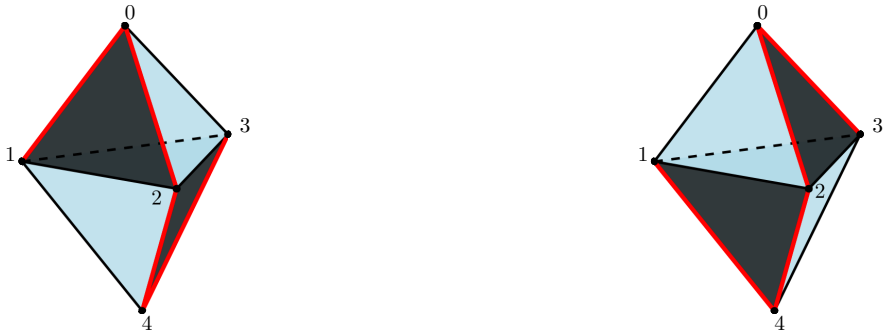} 
    \captionof{figure}{Two of the rooted forests of $\Lambda_4(K)$. The two darkened $2$-cells in both cases indicate simplicies that are not in the forest. The red $1$-cells are the roots. Cyclically permuting the labeling of the vertices $1$, $2$, and $3$ around the equator in both shown rooted forests gives rise to all six elements of $\Lambda_4(K)$.}
    \label{sym}
\end{center}
\end{example}


\section{Future Directions}

We have come to notice through close observations that gradients within discrete Morse theory behave particularly nicely, e.g. satisfy Theorem \ref{main}, when the complex is indeed a triangulation of manifold that may also have had singularities. These complexes go by the name of simplicial \textit{pseudomanifolds}, satisfying the three following properties: for a $d$-dimensional simplicial complex $K$,
\begin{itemize}
    \item The closure of $|K|$ equals a union of $d$-cells.
    \item Each $(d-1)$-cell of $K$ is a face of exactly two $d$-cells of $K$.
    \item Given two $d$-cells $\sigma, \sigma'$ of $K$, there is a sequence of $d$-cells of $K$
    $$\sigma = \sigma_0, \sigma_1, \dots, \sigma_k = \sigma'$$
    such that $\sigma_i \cap \sigma_{i+1}$ is an $(d-1)$-cell for each $i$.
\end{itemize}

We have evidence then to believe the following conjecture is true and intend on releasing a future work addressing this connection in further detail.

\begin{conjecture}
Let $K$ be a $d$-dimensional simplicial complex. Then $K$ is an orientable pseudomanifold if and only if
\begin{equation*} \tag{$\dagger$} \label{maineq}
    det( \Delta_d + \lambda I) = \sum_{i=0}^{n} |\mathcal{M}_{m-i} (K)| \lambda^{n-i},
\end{equation*}
where $m = |K_d|$ and $n = |K_{d-1}|$.
\end{conjecture}

For instance, one can observe that the failure of the M\"obius strip (Example \ref{mobfail}) to be an issue with orientability, whereas the failure of the bipyramid (Example \ref{bipyramid}) as an issue with branching, e.g. the existence of subcomplexes of the forest(s) pictured in the figure above having three $d$-cells share a single $(d-1)$-cell. One may also wonder: if the characteristic polynomials play nicely with gradients in these cases, what else holds true from a discrete Morse theory perspective of pseudomanifolds?

\

We have also had many other interesting questions raised to us, which we invite the reader to explore. For instance, there are many variants of discrete Morse theory that have similar concepts of gradients. One may instead focus on general posets \cite{genposets}, consider instead acyclic matchings on free chain complexes (dubbed \textit{algebraic} Morse theory) \cite{algmorsetheory}, piecewise linear Morse theory (see \cite[Chapter 5]{smoothanddiscrete}), among others. 

\begin{question}
Does a similar statement to $($\ref{maineq}$)$ exist and appear in these other variants of discrete Morse theory?
\end{question}

Discrete Morse theory has been used in various software relating to topological data analysis, such as the \textit{Perseus Software Project for Rapid Computation of Persistent Homology} and the \textit{Computation Homology Project} (CHomP), and one interesting future suggestion is to use this work to pave the way for a deeper understanding of the structure of gradients and allow for improvements in the efficiency of these computations, e.g. developing more optimized algorithms for finding acyclic matchings (see \cite{complexity}).

\begin{question}
Given the indicated structure of subsets of gradients as rooted forests, could this assist in creating new optimized algorithms for determining gradients on simplicial complexes satisfying particular constraints?
\end{question}

On the note of Remark \ref{matroids}, we are also curious about the following.

\begin{question}
Is there a matroid interpretation of $($\ref{maineq}$)$?
\end{question}

Further, recall that the set of independent sets $\mathcal{I}$ of a matroid $M$ form a simplicial complex (the \textit{independence complex} of the matroid).

\begin{question}
If $M=(E,\mathcal{I})$ is a matroid, what do gradients on $\mathcal{I}$ reveal about $M$? In the case of a simplicial matroid $M_K$ for a complex $K$, how does $\det(\Delta_K+\lambda I)$ relate to $\det(\Delta_\mathcal{I}+\lambda I)$?
\end{question}

The Laplacian we discussed here is only ``half" of what is known as the \textit{combinatorial} Laplacian, defined as $\mathcal{L}_n := \partial_{n+1} \partial_{n+1}^T + \partial_n^T \partial_n$ for a $d$-dimensional simplicial complex, where $n\leq d-1$. Though we were unable to find a connection between discrete gradient vector fields and $\det(\mathcal{L}_n+\lambda I)$, it is suspected one may exist. This formulation of the Laplacian is particularly interesting, as for instance, one has that $H_n(K) \cong \ker(\Delta_n)$ due to the combinatorial Hodge decomposition theorem. 

Lastly, there exists a connection between discrete Morse theory and Adinkras: a decorated graph originating from the study of supersymmetric algebras in particle physics (see \cite{adinkra}). In particular, the height function on an Adinkra induces discrete Morse functions for surfaces that are supported on Adinkra graphs.

\begin{question}
What can be said about Adrinkas with the aid of discrete gradient vector fields and their connection to the Laplacian?
\end{question}



\

\nocite{*}

\printbibliography

\end{document}